\documentclass[12pt]{article}

\usepackage{amssymb,amscd,amsmath, amsthm}
\usepackage[pdftex]{graphicx,color}  
\newcommand{\comment}[1]{\marginpar{\sffamily{\tiny #1
\par}\normalfont}}
\renewcommand{\comment}[1]{}

\def\atom{\mathop{\rm atom}}
\def\coat{\mathop{\rm coat}}

\def\Aa{{\cal A}}

\def\Ff{{\cal F}}
\def\Gg{{\cal G}}
\def\corank{\mathop{\rm corank}}

\def\rank{\mathop{\rm rank}}
\def\sign{\mathop{\rm sign}}

\newtheorem{thm}{Theorem}[section]
\newtheorem{lemma}[thm]{Lemma}
\newtheorem{prop}[thm]{Proposition}

\newtheorem{cor}[thm]{Corollary}
\newtheorem{defn}[thm]{Definition}
\newtheorem{nota}[thm]{Notation}

\newenvironment{numpf}[1]{{\it #1}}{\hfill$\Box$}
\newtheorem*{TRT}{Topological Representation Theorem}
\newtheorem*{QTA}{Quillen's Theorem A}
\def\beginwval#1#2{\bgroup% change to \the#1 should be local
  \edef\@savecount{\the\value{#1}}% save old value of counter
  \expandafter\def\csname the#1\endcsname{#2}% new value
  \begin{#1}}% begin environment
\def\endwval#1{\setcounter{#1}{\@savecount}\end{#1}\egroup}

\begin{document}
\title{Homotopy sphere representations for matroids}
\author{Laura Anderson}
\date{}
\maketitle
\begin{abstract} For any rank $r$ oriented matroid $M$, a construction is given of a "topological representation" of $M$ by an  arrangement of homotopy spheres in a simplicial complex which is homotopy equivalent to $S^{r-1}$. The construction is completely explicit and depends only on a choice of maximal flag in $M$. If $M$ is orientable, then all Folkman-Lawrence representations of all orientations of $M$ embed in this representation in a homotopically nice way. 
\end{abstract}

A fundamental result in oriented matroid theory is the Topological Representation Theorem (\cite{FL}), which says that every rank $r$  \emph{oriented} matroid can be represented by an arrangement of oriented pseudospheres in $S^{r-1}$. In~\cite{Ed} Swartz made the startling discovery that any rank $r$ \emph{matroid} can be represented by an arrangement of homotopy spheres in a $(r-1)$-dimensional CW complex homotopic to $S^{r-1}$. The representation is far from canonical: it depends on, among other things, a choice of tree for each rank 2 contraction and choices of
cells glued in to kill off homotopy groups.

The present paper, inspired by Swartz's work, gives a topological
representation of any rank $r$ matroid by an arrangement of homotopy spheres in a simplicial complex which is homotopy equivalent to $S^{r-1}$. The construction is completely explicit and depends only on a choice of maximal flag. For oriented matroids, there is a nice homotopy relationship between this representation and the representation given by the Topological Representation Theorem of Folkman and Lawrence.

\section{Matroids background}

There are many equivalent characterizations of matroids. We will define matroids in terms of {\em flats}.

\begin{defn} A \textbf{geometric lattice} is a finite ranked poset $L$ such 
that:
\begin{enumerate}
\item $L$ is a lattice, i.e., every pair $X,Y$ of elements has a meet $X\wedge Y$ and a join $X\vee Y$. (In particular, $L$ has a least element, denoted $\hat 0$, and a greatest element, denoted $\hat 1$.)
\item Every element of $L$ is a join of atoms of $L$.
\item The rank function is semimodular, i.e., for any $X$ and $Y$ in $L$, $\rank(X)+\rank(Y)\geq \rank (X\wedge Y)+\rank (X\vee Y)$.
\end{enumerate}
\end{defn}

\begin{defn} Let $E$ be a finite set. A rank $r$ \textbf{matroid} on $E$ is a collection
$M$ of subsets of $E$, viewed as a poset ordered by inclusion, such that 
\begin{itemize}
\item $M$ is a rank $r$ geometric lattice,
\item the meet of any two elements is their intersection, and
\item $E$ is the join of all the atoms of $L$.
\end{itemize}
 The elements of $M$ are called  \textbf{flats}. 
\end{defn}

The canonical example arises from a finite arrangement $\{H_e:e\in E\}$ of hyperplanes (in an arbitrary vector space). Define a \textbf{flat} of the arrangement to be $A\subseteq E$ such that, for every $e\in E-A$, $(\cap_{a\in A}H_a)\cap H_e\neq \cap_{a\in A}H_a$. Then the set of flats of the arrangement is a matroid on $E$.

We note for future reference some standard facts about matroids:
\begin{nota} If $L$ is a lattice and $X, Y\in L$ then 
\begin{enumerate}
\item $\coat(X)$ denotes the set of coatoms of $L$ which are greater than or equal to $X$.
\item $L_{\geq X}$ denotes $\{Y\in L: Y\geq X\}$, and $L_{\leq X}$ denotes $\{Y\in L: Y\leq X\}$.
\item $\coat_{L_{\geq Y}}(X)$ denotes the set of coatoms of $L_{\geq Y}$ which are greater than or equal to $X$.
\end{enumerate}
\end{nota}

\begin{lemma}\label{lem:coat} (cf. Proposition 3.4.2 in~\cite{faigle}) 1. Every interval in a geometric lattice is a geometric lattice.

2.  If $L$ is a geometric lattice and $\hat 1\neq X\in L$ then $X=\wedge\coat(X)$.
\end{lemma}

\begin{defn} A \textbf{complete flag}  in a matroid $L$ is a maximal chain in $L$.
\end{defn}

\begin{lemma} \label{lem:flag} (cf. Theorem 3.3.2 in~\cite{faigle}) Let $L$ be a geometric lattice.  If $\Ff$ is a maximal chain in $L$ and $X\in L$ then
$\{X\vee F| F\in\Ff\}$ is a maximal chain in $L_{\geq X}$ and $\{X\wedge F| F\in\Ff\}$ is a maximal chain in $L_{\leq X}$ .
\end{lemma}

\section{Homotopy arrangements} 
\begin{defn} \label{defn:hom rep} A \textbf{homotopy $d$-arrangement} is a simplicial complex $S$ and
a finite set $\Aa$ of subcomplexes of $S$ such that:
\begin{enumerate}
\item $S\simeq S^d$
\item Each $T\in\Aa$ is homotopy equivalent to $S^{d-1}$.
\item Every intersection of elements of $\Aa$ is homotopy equivalent to a sphere.
\item There is a free $\mathbb Z_2$ action on $S$ that restricts to an action on each $T\in\Aa$.
\item If $U\simeq S^{d'}$ is an intersection of elements of $\Aa$ and $T\in \Aa$ such that $U\not\subseteq T$ then $U\cap T\simeq S^{d'-1}$. 
\end{enumerate}
\end{defn}

\begin{defn} Let $(S, \Aa)$ be a homotopy $d$-arrangement, where $\Aa=\{T_e:e\in E\}$. A subset $F$ of $E$ is a \textbf{flat} if, for every $e\in E-F$,  $(\cap_{f\in F}T_f)\bigcap T_e\neq \cap_{f\in F}T_f$.
\end{defn}

\begin{prop}\label{prop:geom latt}  If $(S, \{T_e:e\in E\})$ is a homotopy $d$-arrangement then the set of all its flats  is a matroid. 
\end{prop}

We say that $(S,\{T_e:e\in E\})$ is a \textbf{homotopy representation} of this matroid. 

The proof of Proposition~\ref{prop:geom latt} is strainghtforward. The more interesting result is the converse, Corollary~\ref{bigcor}:  every matroid has a homotopy representation.

\begin{proof} That the poset of flats is ranked follows easily from induction and Property 5 in Definition~\ref{defn:hom rep}. The only other nontrivial thing to check is semimodularity. Let $F$ and $G$ be flats with $\cap_{e\in F}T_e\simeq S^f$ and $\cap_{e\in G}T_e\simeq S^g$. Let $\{T_{e_1}, \ldots, T_{e_k}\}$ be a subset of $F-G$ such that, for every $j\in[k]$, $\bigcap_{e\in F\cap G} T_e\cap\bigcap_{i=1}^{j-1} T_{e_i}\not\subseteq T_{e_j}$ and 
$\bigcap_{e\in F\cap G} T_e\cap\bigcap_{i=1}^{k} T_{e_i}=\cap_{e\in F} T_e$. Then, by Property 5 in Definition~\ref{defn:hom rep},
$$\cap_{e\in F\wedge G}T_e=\cap_{e\in F\cap G}T_e\simeq S^{f+k}$$
and 
$$\cap_{e\in F\vee G}T_e=\cap_{e\in F\cup G}T_e=(\cap_{e\in G}T_e)\cap(\cap_{i=1}^k T_{e_i})\simeq S^h,$$
where $h\geq g-k$. (Here wedge and join are taken in the lattice of flats of $\Aa$.) Thus 
$$\begin{array}{rcl} 
\rank({F\wedge G})+\rank({F\vee G})&=&r-(f+k)+r-h\\
&\leq&r-f+r-g\\
&=&\rank({F})+\rank({G}).
\end{array}$$
\end{proof}

\section{The construction}

Throughout the following $L$ is a geometric lattice of rank $r$.  We will construct a simplicial complex with two vertices $G_+, G_-$ for each coatom $G$ of $L$.

Fix a complete flag ${\Ff}=\{\hat{0} =F_0<\cdots<F_r=\hat{1}\}$ in $L$. For every $i$, let
$A_i=\coat(F_i)\backslash\coat(F_{i+1})$.

Lemma~\ref{lem:coat} tells us that each $A_i$ is nonempty. Note that $\{A_0,\ldots,A_{r-1}\}$ is a partition of the coatoms of $L$.

Define $S_{\hat 0}$ to be the simplicial complex with $2^r$ maximal simplices
$$V_0\cup\cdots \cup V_{r-1}$$
where each $V_i\in \{\{G_+:G\in  A_i\}, \{G_-:G\in A_i\}\}$.

More generally, for each $G\in L$, let $S_G$ be the simplicial complex with maximal simplices 
$$W_0(G)\cup\cdots\cup W_{r-1}(G)$$
where each $W_i(G)\in \{ \{H_+:H\in \coat(G)\cap A_i\},  \{H_-:H\in \coat(G)\cap A_i\}\}$. For every such maximal simplex $\sigma$ and every $i$, call $W_i(G)$ the \textbf{$i$th part of $\sigma$}. We describe $\sigma$ by a vector $\sign(\sigma)\in \{+,-,0\}^{[0, r-1]}$ by defining 
$$(\sign(\sigma))_i=\left\{\begin{array}{cc}
0 &\mbox{ if $W_i(G)=\emptyset$}\\
+ &\mbox{ if $W_i(G)=\{H_+:H\in \coat(G)\cap A_i\}\neq\emptyset$}\\
- &\mbox{ if $W_i(G)=\{H_-:H\in \coat(G)\cap A_i\}\neq\emptyset$}
\end{array}\right.$$

Notice that $S_{\hat 1}=\emptyset$. Also, for every $G<H$ in $L$, $S_G\supset S_H$. In particular, every $S_G$ is a subcomplex of $S_{\hat 0}$.

When we wish to emphasize the dependency on $\Ff$, we will denote $A_i$ by $A_i(\Ff)$ and $S_G$ by $S_G(\Ff)$.

\noindent{\bf Example.} For the rank 2 uniform matroid on $E=[4]$, let ${\Ff}=\{\hat 0<\{1\}<[4]\}$. Then $S_{\hat 0}({\Ff})$ consists of four tetrahedra and their faces, as shown in Figure~\ref{fig:rank2}. For each rank 1 flat $\{x\}$, $S_{\{x\}}$ consists of the two vertices $\{x\}_+, \{x\}_-$.
\begin{figure}
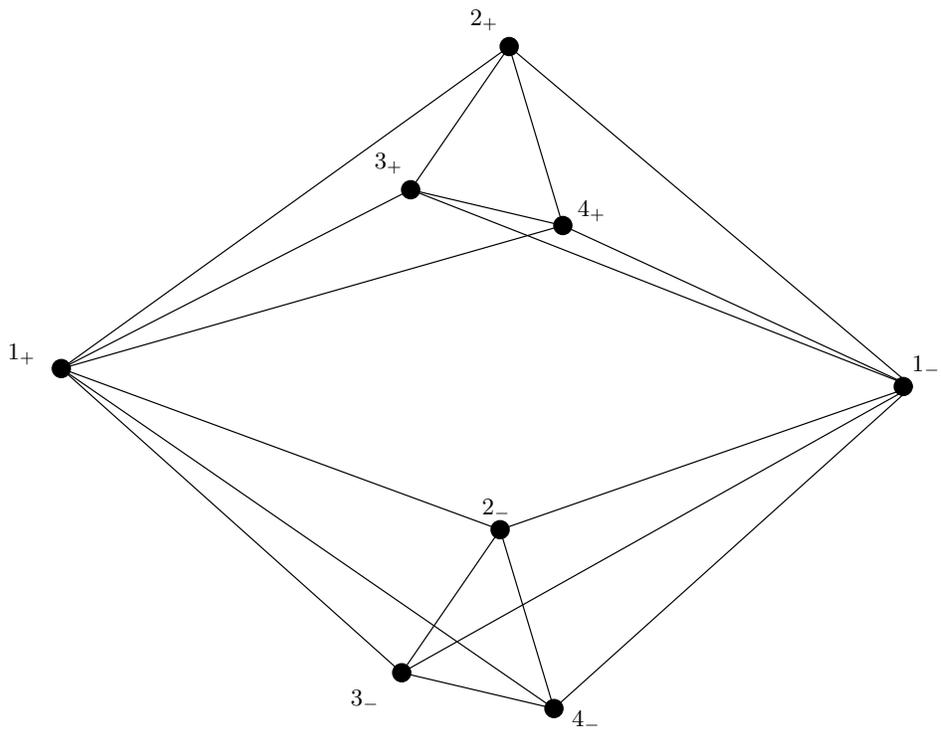
\caption{The homotopy sphere $S_{\hat 0}$for the rank 2 uniform matroid on 4 elements\label{fig:rank2} }
\end{figure}

\begin{thm}\label{thm:main}
\begin{enumerate}
\item $(S_{\hat 0}, \{S_G:G\in\atom(L)\})$ is a homotopy $(r-1)$-arrangement.
\item For every $G\in L$, $S_G\simeq S^{\corank(G)-1}$
\item For every $G, H\in L$, $S_G\cap S_H=S_{G\vee H}$
\end{enumerate}
\end{thm}

\begin{cor} \label{bigcor} For every matroid $M$ and complete flag $\Ff$ in $M$, $(S_{\hat{0}}(\Ff), \{S_G: G\mbox{ a minimal non-$\hat{0}$ flat of }M\})$ is a homotopy representation of $M$.
\end{cor}
 
Before we begin the proof, some notation is in order. For every $G\in L$, let $c(G)=\{i\in[0,r-1]:\coat(G)\cap A_i\neq\emptyset\}$. Thus $c(G)$ is exactly the set of all $i$ for which $(\sign(\sigma))_i\neq \vec{0}$ for maximal simplices $\sigma$ of $S_G$.

We will view $\{+,-,0\}$ as a poset with maximal elements $+$ and $-$ and minimal element $0$, and we will order $\{+,-,0\}^{[0, r-1]}$ componentwise.

For every $\vec{v}\in \{+,-,0\}^{[0, r-1]}$ and every $G\in L$, let $\sigma(\vec{v}, G)=V_0\cup\cdots\cup V_{r-1}$, where 
$$V_i= \left\{\begin{array}{ll}
\{H_+:H\in\coat(G)\cap A_i\}&\mbox{ if $v_i=+$}\\
\{H_-:H\in\coat(G)\cap A_i\}&\mbox{ if $v_i=-$}\\
\emptyset&\mbox{ if $v_i=0$}
\end{array}\right.$$

For any $\vec{v}\in\{+,-,0\}^{[0, r-1]}-\{\vec{0}\}$, let $\vec{v}_*$ denote the vector obtained from $\vec{v}$ by deleting all 0 components.

Also before we begin the proof, we briefly review the \textbf{cross-polytope}. For our purposes it is convenient to view the cross-polytope in ${\mathbb R}^{[0, r-1]}$. It is defined to be the convex hull
${\mathrm{conv}}(\{\vec{e_i}:i\in[0,r-1]\} \cup\{-\vec{e_i}:i\in[0,r-1]\})$ of the coordinate vectors  and their negatives. Thus it is an $r$-dimensional simplicial convex polytope whose poset of proper faces is isomorphic to $\{+,-,0\}^{[0, r-1]}$. 

The  gist of the proof is that, for every element $G$ of $L$, there is a canonical isomorphism from the nerve of the set of maximal simplices of $S_G$ to the nerve of the set of maximal proper faces of the cross-polytope of dimension $\corank(G)$. Thus, by Lemma~\ref{lem:intlattice}, $S_G\simeq S^{\corank(G)-1}$.

\begin{numpf}{Proof of Theorem \ref{thm:main}}
Proof of (2): We  first show that there are exactly $\corank(G)$ elements $i$ of $[0, r-1]$ such that $\coat(G)\cap A_i\neq\emptyset$. Notice that, for any $i$, 
$$\begin{array}{rcl}
\coat(G)\cap A_i&=&\coat(G)\cap (\coat(F_i)\backslash\coat(F_{i+1}))\\
&=&\coat_{M_{\geq G}}(F_i)\backslash \coat_{M_{\geq G}}(F_{i+1})\\
&=&\coat(G\vee F_i)\backslash\coat(G\vee F_{i+1}).\end{array}$$
 Thus $\coat(G)\cap A_i\neq\emptyset$ if and only if $G\vee F_i<G\vee F_{i+1}$. By Lemma~\ref{lem:flag}, there are exactly $\rank(M_{\geq G})=\corank(G)$ values $i$ such that $G\vee F_i<G\vee F_{i+1}$.

Thus the maximal simplices of $S_G$ are in bijection with the maximal elements of $\{+,-,0\}^{[0,\corank(G)-1]}$ by the bijection $\sigma\leftrightarrow \vec{v}(\sigma)_*$. In fact, we'll show, for every set $\sigma_1,\ldots, \sigma_n$ of maximal simplices of $S_G$,
$$\cap_{i=1}^n\sigma_i=\sigma(\bigwedge_{i=1}^n\vec{v}(\sigma_i), G).$$

This is easily seen by noting that each side of the equation has the same $j$th part:
$$\sign{(\cap_{i=1}^n\sigma_i)}_j=
\left\{\begin{array}{ll}
+&\mbox{ if $v_j(\sigma_i)=+$ for all $i$}\\
-&\mbox{ if $v_j(\sigma_i)=-$ for all $i$}\\
0&\mbox{ otherwise}
\end{array}\right\}=
\sign(\sigma(\bigwedge_{i=1}^n \vec{v}(\sigma_i), G))_j.$$

Notice that $\cap_{i=1}^n\sigma_i$ is empty if and only if $\bigwedge_{i=1}^n\vec{v}(\sigma_i)=\vec{0}$, and two such intersections $\cap_{i=1}^n\sigma_i$ and $\cap_{i=1}^n\tau_i$ coincide if and only if $\bigwedge_{i=1}^n\vec{v}(\sigma_i)=\bigwedge_{i=1}^n\vec{v}(\tau_i)$. Thus the set of maximal simplices of $S_G$ has nerve isomorphic to that of the set of maximal proper faces of a $\corank(G)$-dimensional cross-polytope. By Lemma~\ref{lem:intlattice} we conclude that $S_G\simeq S^{\corank(G)-1}$.

Proof of (3): The maximal simplices of $S_G\cap S_H$ are exactly the simplices of the form $\alpha\cap\beta$, where $\alpha$ is a maximal simplex of $S_G$, $\beta$ is a maximal simplex of $H$, and $\{\vec{\alpha}_i, \vec{\beta}i\}\neq\{+,-\}$ for every $i$. In other words, the maximal simplices of $S_G\cap S_H$ are exactly the simplices of the form $\sigma(\vec{v_\alpha}, G)\cap \sigma(\vec{v_\beta}, H)$ with $(\vec{v_\alpha})_i=(\vec{v_\beta})_i$ for each $i\in c(G)\cap c(H)$.
We will show these are exactly the maximal simplices of $S_{G\vee H}$.

Note that $\coat(G\vee H)=\coat(G)\cap \coat(H)$, and so $c(G\vee H)=c(G)\cap c(H)$. So the maximal simplices of $S_{G\vee H}$ are all $\sigma(\vec{v}, G\vee H)$ where $\vec{v}$ has non-0 components $c(G)\cap c(H)$. If $\vec{v}$ is such a vector and $\vec{v'}>\vec{v}$ then $\sigma(\vec{v}, G\vee H)=\sigma(\vec{v'}, G\vee H)$. In particular, for every such $v$ we can choose $\vec{v_\alpha}>\vec{v}$ with non-0 components $c(G)$ and $\vec{v_\beta}>\vec{v}$ with non-0 components $c(H)$, and then $\sigma(\vec{v}, G\vee H)=\sigma(\vec{v_\alpha}, G\vee H)\cap \sigma(\vec{v_\beta}, G\vee H)$. This is easily seen to be $\sigma(\vec{v_\alpha}, G)\cap \sigma(\vec{v_\beta}, H)$.

Proof of (1): $S_{\hat 0}\simeq S^{r-1}$ by Lemma~\ref{lem:intlattice} as in the proof of (2). The second and third parts of the definition are proved by (2) and (3). The $\mathbb Z_2$ action is simply to send each vertex $G_+$ to $G_-$ and vice-versa. 

To see the final part of the definition, note that by (2) an intersection of elements of $\{S_G:G\in\atom(L)\}$ is just $S_H$ for some $H\in L$. Let $G\in\atom(L)$ such that $S_H\not\subseteq S_G$. Then $H\not\geq G$, and $S_H\cap S_G=S_{H\vee G}\simeq S^{\corank(H\vee G)-1}$. By semimodularity $\rank(G\vee H)\leq\rank(G)+\rank(H)-\rank (G\wedge H)=1+\rank(H)-0$, and since $H\not\geq G$ we get  $\rank(G\vee H)=1+\rank(H)$.
 
\end{numpf}

\section{Oriented matroids}

The inspiration for this paper came from Swartz's similar result (\cite{Ed}), which in turn was inspired by Folkman and Lawrence's Topological Representation Theorem (\cite{FL}). Folkman and Lawrence's result concerns oriented matroids, which are combinatorial analogs to arrangements of oriented hyperplanes in ${\mathbb R}^n$ (or ${\mathbb F}^n$, where $\mathbb F$ is any field of characteristic 0). Their result (given precisely later in this section) says that any oriented matroid can be represented by an arrangement of oriented "topological equators" in $S^{n-1}$. Thus it states that, topologically, oriented matroids constitute a very good model for real hyperplane arrangements. Swartz's result (and the present Corollary~\ref{bigcor}) are more mysterious. For instance, why should a matroid arising from an arrangement of hyperplanes in ${\mathbb Z}_2^n$ have a representation by homotopy spheres?

The situation could be seen as less mysterious if our representations of matroids were wholly unrelated to the representations of Folkman and Lawrence. However, this section will prove that, for an orientable matroid $M$ with a choice of complete flag, the Folkman-Lawrence topological representation of any orientation of $M$ embeds in a topologically nice way in the homotopy representation of $M$. 

There is no similar result for Swartz's homotopy representations. For instance,  for rank 2 matroids, Swartz's representations are 1-dimensional homotopy circles, which precludes the existence of  embeddings with the properties we want. (As an aside, it should be noted that Swartz's construction has virtues that the construction of the present paper lacks -- most obviously, his homotopy representations of rank $r$ matroids are $(r-1)$-dimensional complexes.)

\subsection{Background} 

Like ordinary matroids, oriented matroids have several equivalent characterizations. An oriented matroid $M$ on a finite set $E$ can be given, for instance, by:
\begin{itemize}
\item Its set $V^*(M)$ of \textbf{covectors}. $V^*(M)$ is defined to be a subset of $\{+,0,-\}^E$ satisfying certain axioms. The axioms are somewhat lengthy and won't be used directly here. (They may be found, for instance, in~\cite{BLSWZ}.) All we need to know is the Topological Representation Theorem, given below.
\item Equivalently, its set $C^*(M)$ of \textbf{cocircuits}. $C^*(M)$ is the set of minimal nonzero elements of $V^*(M)$. (Here $V^*(M)$ is viewed as a  subposet of $\{+,0,-\}^E$.)
\end{itemize}

The motivating example: let $\Aa=\{\vec{v_e}^\perp:e\in E\}$ be a finite arrangement of hyperplanes in ${\mathbb R}^n$. $\Aa$ defines a function $s:{\mathbb R}^n\to\{+,-,0\}^E$ by $s(\vec{x})(e)=\sign(\vec{x}\cdot\vec{v_e})$. The image of $s$ is the set of covectors of an oriented matroid. Thus there is a bijection between covectors of the oriented matroid and cones in the partition of ${\mathbb R}^n$ determined by $\Aa$.

As the motivating example suggests, oriented matroids encode strictly more data than matroids. Specifically:
\begin{defn} If $M$ is an oriented matroid, then $L:=\{X^{-1}(0):X\in V^*(M)\}$ is a matroid, called the \textbf{underlying matroid} of $M$. The set of coatoms of $L$ is 
 $\{X^{-1}(0):X\in C^*(M)\}$. For each coatom $C$ of $L$ there are exactly two cocircuits $X, -X$ of $M$ such that $C=X^{-1}(0)$.  
 
 For every $G\in L$, the corresponding  \textbf{covector flat} of $M$ is $\{X\in V^*(M):X(G)=0\}$. Denote this set by $M_G$.
\end{defn}

 \begin{defn} A subset $S$ of $S^{r-1}$ is a \textbf{pseudosphere} if there is an automorphism of $S^{r-1}$ taking $S$ to an equator.
 \end{defn}
 
 Thus, if $S\subset S^{r-1}$ is a pseudosphere then $S^{r-1}\backslash S$ has two connected components.
  \begin{defn} An  \textbf{oriented pseudosphere} in $S^{r-1}$ is a triple $(S, S^+, S^-)$, where $S$ is a pseudosphere in $S^{r-1}$ and $S^+$, $S^-$ are the two components of $S^{r-1}\backslash S$.

 A rank $r$ \textbf{arrangement of oriented pseudospheres} in $S^{r-1}$ is a finite set $\{(S_e, S_e^+, S_e^-):e\in E\}$ of oriented pseudospheres in $S^{r-1}$ satisfying:
 \begin{enumerate}
 \item For all $A\subseteq E$, $S_A:=\cap_{e\in E}S_e$ is a topological sphere.
 \item If $A\subseteq E$ and $e\in E$ such that $S_A\not\subseteq S_e$ then $(S_e\cap S_A, S_e^+\cap S_A, S_e^-\cap S_A)$ is an oriented pseudosphere in $S_A$.
 \end{enumerate}
 \end{defn}
 
 Loosely put, an arrangement of oriented pseudospheres in $S^{r-1}$ is a finite set of oriented topological equators that behaves topologically like a finite set of oriented equators.
 
 An arrangement $\Aa=\{(S_e, S_e^+, S_e^-):e\in E\}$ of oriented pseudospheres in $S^{r-1}$ determines a function $\sign:S^{r-1}\to\{+,0,-\}^E$, where 
 $$\sign(x)(e)=\left\{\begin{array}{ll}
 0&\mbox{ if $x\in S_e$}\\
 +&\mbox{ if $x\in S_e^+$}\\
-&\mbox{ if $x\in S_e^-$}
\end{array}\right.$$
Let $V^*(\Aa)=\{\sign(X):x\in S^{r-1}\}\cup\{0\}$. Perhaps the most fundamental result in the theory of oriented matroids is:
 \begin{TRT}  (\cite{FL}) For any arrangement $\Aa$ of oriented pseudospheres in $S^{r-1}$, $V^*(\Aa)$ is the set of covectors of an oriented matroid.
 
 Conversely, for any rank $r$ oriented matroid $M$ on a set $E$, $V^*(M)=V^*(\Aa)$ for some  arrangement $\Aa$ of oriented pseudospheres in $S^{r-1}$.
\end{TRT}

\subsection{Topological representations and homotopy representations}

 In the following a simplicial complex is viewed as a poset, ordered by inclusion. For any poset $P$, we will denote its order complex by $\Delta P$.
 
\begin{prop}\label{prop:orientable} If $L$ is the lattice of flats of an orientable matroid and $M$ is an orientation of this matroid, then for each complete flag ${\Ff}=\{\hat{0}=F_0<F_1<\cdots<F_r=\hat{1}\} $ in $L$ and each choice of $e_1, \ldots, e_{r-1}$ with $e_i\in F_i\backslash F_{i-1}$ there is a canonical injection of posets $\iota: V^*(M)-\{0\}\to S_{\hat 0}$ taking each $M_G$ into $S_G$. Further, this injection induces a homotopy equivalence $\Delta V^*(M)\to \Delta S_{\hat 0}$ and restricts to a homotopy equivalence $\Delta M_G\to \Delta S_G$ for every $G$, and the $\mathbb Z_2$ action on $S_{\hat 0}$ restricts to the usual $\mathbb Z_2$ action $\iota(X)\to \iota(-X)$ on $\iota(V^*(M))$.
\end{prop}

 The geometric idea of this injection is quite simple. We first describe it informally on cocircuits of $M$. Every antipodal pair $\{X,-X\}$ of cocircuits of $M$ corresponds to exactly one coatom $X^{-1}(0)$ of $L$, and hence to exactly two vertices $X^{-1}(0)_+, X^{-1}(0)_-$ of $S^{\hat 0}(\Ff)$. Our injection $\iota$
\begin{itemize}
\item will map a cocircuit $X$ in $e_1^+$ to $X^{-1}(0)_+$,
\item will map a cocircuit $X$ in $e_1^-$ to $X^{-1}(0)_-$,
\item will map a cocircuit $X$ in $e_1^0\cap e_2^+$ to $X^{-1}(0)_+$,
\item will map a cocircuit $X$ in $e_1^0\cap e_2^-$ to $X^{-1}(0)_-$,
\end{itemize}
and so forth. Notice that the first two cases describe $X$ with $X^{-1}(0)\in A_1$, the next  two cases describe $X$ with $X^{-1}(0)\in A_2$, and so forth.
Then we  define the image of any $X\in V^*(M)-\{0\}$ to be the join of the images of the cocircuits less than or equal to $X$. (We must prove that this join exists.)

So for each pseudohemisphere $P=e_1^0\cap e_2^0\cap\cdots\cap e_{k-1}^0\cap e_k^\epsilon$ with $\epsilon\in\{+,-\}$, $\iota$ sends the cocircuits in $P$ to the vertices of the simplex $\{\epsilon_c:c\in A_k\}$. Thus, it sends the "obvious" contractible subsets of $V^*(M)-\{0\}$ suggested by $\Ff$ to the "obvious" contractible subsets of $S_{\hat 0}$. This raises hope of a nerve argument showing $\iota$ to be a homotopy equivalence, which is indeed what we'll find.
   
\begin{lemma}\label {lem:basis} For any $i$, let $F_i$ and $e_i$ be as in Proposition~\ref{prop:orientable}. Then $\coat(F_i)=\{X^{-1}(0):X\in C^*(M)\mbox{ and }\{e_1, \ldots, e_i\}\subseteq X^{-1}(0)\}$.
\end{lemma}

\begin{proof} By definition of the $e_j$, for any $X^{-1}(0)\in\coat(F_i)$ we have $\{e_1,\ldots, e_i\}\subseteq F_i\subseteq X^{-1}(0)$. Conversely, by standard matroid results we know $M/{F_i}=\{Y\in V^*(M):\{e_1, \ldots, e_i\}\subseteq Y^{-1}(0)\}$. Thus if $\{e_1, \ldots, e_i\}\subseteq X^{-1}(0)$ then $X\in M/{F_i}$, and so $X^{-1}(0)\supseteq F_i$, and so $X\in\coat(F_i)$.
\end{proof}

The proof of Proposition~\ref{prop:orientable} will use two operations on sign vectors:
\begin{defn} Let $X\in\{+,0,-\}^E$.

1. If $F\subset E$, define $X/F\in\{+,0,-\}^E$ by
$$X/F(e)=\left\{\begin{array}{ll}
X(e)&\mbox{ if $e\not\in F$}\\
0&\mbox{ if $e\in F$}
\end{array}\right.$$

2. If $f\not\in E$ and $\epsilon\in\{+,0,-\}$, define $X\circ f^\epsilon\in\{+,0,-\}^{E\cup\{f\}}$ to be the extension of $X$ with $X(f)=\epsilon$.
\end{defn}

\begin{numpf} {Proof of Proposition~\ref{prop:orientable} }  
For any $X\in C^*(M)$, let $i_X=\min\{i:X(e_i)\neq 0\}$. Define $\iota(X)=\{(X(e_{i_X}))_{X^{-1}(0)}\}$.
For any $Y\in (V^*(M)-\{0\})\backslash C^*(M)$, let $$\iota(Y)=\bigcup_{X\in C^*(M)_{<Y} } \iota(X).$$
We first check that $\iota(Y)$ is a simplex in $S_{Y^{-1}(0)}$. For this we need to check:
\begin{itemize}
\item Each vertex of $\iota(Y)$ is in $S_{Y^{-1}(0)}$. This is immediate: if $X\in C^*(M)_{<Y}$ then $X^{-1}(0)\in\coat(Y^{-1}(0))$, so $X^{-1}(0)\in\coat(Y^{-1}(0))\cap A_i$ for some $i$.
\item All vertices of $\iota(Y)$ are in a common simplex of $S_{Y^{-1}(0)}$. That is,  if $\{X_1, X_2\}\subseteq C^*(M)_{<Y}$ and $\{X_1^{-1}(0), X_2^{-1}(0)\}\subseteq\coat(Y^{-1}(0))\cap A_i$ for some $i$ then $\iota(X_1)_{X_1^{-1}(0)}=\iota(X_2)_{X_2^{-1}(0)}$.  To see this: since $\{X_1^{-1}(0), X_2^{-1}(0)\}\subseteq A_i$, we have that $F_i\subseteq X_i^{-1}(0)\cap X_2^{-1}(0)$ and $F_{i+1}\not\subseteq X_i^{-1}(0)\cap X_2^{-1}(0)$. Thus, by Lemma~\ref{lem:basis}
$\{e_1, \ldots, e_i\}\subseteq X_i^{-1}(0)\cap X_2^{-1}(0)$ and $e_{i+1}\not\in X_i^{-1}(0)\cap X_2^{-1}(0)$. So $i+1=i_{X_1}=i_{X_2}$ and so $\iota(X_1)=(X_1(e_{i+1}))_{X_1^{-1}(0)}$ and $\iota(X_2)=(X_2(e_{i+1}))_{X_2^{-1}(0)}$. Since $X_1$ and $X_2$ are both less than $Y$, $X_1(e_{i+1})\neq 0$, and $X_2(e_{i+1})\neq 0$, we have $X_1(e_{i+1})=X_2(e_{i+1})$.
\end{itemize}
 Thus $\iota$ is well-defined and takes each $M_G$ to $S_G$. $\iota$ is clearly a poset injection.

The homotopy equivalences follow from Lemma~\ref{carrier}, applied to each restriction $\iota_G: \Delta (M_G-\{0\})\to \Delta (S_G)$. For every $G\in L$ and $\vec{v}\in\{+,0,-\}^r$, let $A_{\vec{v}}(G)=\{X\in V^*(M_G):X(e_{i_X})=v_{i_X}\}$. Let
$B_{\vec{v}}(G)$ be the full simplex on $\{(v_i)_C: v_i\neq 0, C\in A_i\cap\coat(G)\}$. 
Our covers  of $V^*(M_G)-\{0\}$ and $S_G$ will be $\{\Delta A_{\vec{v}}(G):\vec{v}\in\{+,-\}^r\}$ resp.$ \{B_{\vec{v}}(G):\vec{v}\in\{+,-\}^r\}$.
The order complexes of the $A_{\vec{v}}(G)$ resp. $B_{\vec{v}}(G)$ will be our covers of $V^*(M_G)-\{0\}$ resp. $S_G$.

Note that, for any $\vec{v},\vec{w}\in\{+,-,0\}^r$, we have $A_{\vec{v}}(G)\cap A_{\vec{w}}(G)=A_{\vec{v}\wedge\vec{w}}(G)$ and $B_{\vec{v}}(G)\cap B_{\vec{w}}(G)=B_{\vec{v}\wedge\vec{w}}(G)$. Also, $A_{\vec{v}}(G)=\emptyset\Leftrightarrow v_i=0$ for every $i$ such that $F_i\not\leq G \Leftrightarrow B_{\vec{v}}(G)=\emptyset$.

To see that the conditions of Lemma~\ref{carrier} are satisfied:

1. We need to see every chain $\gamma=\{X_1<\cdots<X_k\}$ in $V^*(M_G)-\{0\}$ is in some $A_{\vec{v}}(G)$. Clearly, the appropriate $\vec{v}$ has components $v_i=\max\{X_j(e):j\in[k]\}$. 

Also, it's clear that every maximal chain in $S_G$ is contained in some $B_{\vec{v}}(G)$.

2. To see that every nonempty intersection of elements of our covers is contractible, by our earlier observations we need only to show each $A_{\vec{v}}(G)$ and each $B_{\vec{v}}(G)$ is contractible. The latter is trivial. The former will be proved by induction on the number of elements of $M$, and within this, in the rank of $M$. The case of rank 1 is trivial. 

The case in which $M$ is a coordinate oriented matroid (i.e., $M$ is the rank $r$ oriented matroid on elements $\{e_1, \ldots, e_r\}$) requires just a little work. We may assume $v_1\neq 0$, since otherwise  $A_{\vec{v}}(G)\subset V^*(M/e_1)$, and so induction on rank covers it.  Let $i_0=\min\{i>1:v_i\neq 0\}$. Then the following sequence of order homotopies (Definition~\ref{defn:orderhom}) retracts $\Delta A_{\vec{v}}(G)$ to $\Delta A_{(0,\ldots, 0, v_{i_0}, v_{i_0+1},\ldots, v_r)}(G)$.
\begin{enumerate}
\item the lowering homotopy $X\to X/\{e_2, \ldots, e_{i_0-1}\}$
\item  the lowering homotopy 
$$X\to \left\{\begin{array}{cc}	X&\mbox{ if $X(e_{i_0})\in \{0, v_{i_0}\}$}\\
	X/\{e_{i_0}\}&\mbox{ if $X(e_{i_0})=-v_{i_0}$}
	\end{array}\right.$$
\item the raising homotopy $$X\to \left\{\begin{array}{cc}
		X\circ e_{i_0}^+&\mbox{ if $X(e_{i_0})=0$}\\
		X&\mbox{ otherwise}
		\end{array}\right.$$
\item  the lowering homotopy $X\to X/\{e_1\}$.
\end{enumerate}
Thus, by our induction on rank, $\Delta A_{\vec{v}}(G)$ is contractible.

Now, assume we have the result for $M\backslash\{f\}$, where $f$ is an element of $M$ not in $\{e_1, \ldots, e_r\}$. Let $d:A_{\vec{v}}(G)\to A_{\vec{v}}(G\backslash\{f\})$ be the map $X\to X\backslash\{f\}$. Then for every $Y\in A_{\vec{v}}(G\backslash\{f\})$, $d^{-1}(Y)$ has a unique minimal element, and so the upper order ideal generated by $d^{-1}(Y)$ is contractible. Thus, by Quillen's Theorem A (see Section~\ref{sec:homotopy}), $d$ is a homotopy equivalence.

3. Clearly $\iota(A_{\vec{v}}(G))\subseteq B_{\vec{v}}(G)$.

 The statement on the $\mathbb Z_2$ action is clear. 
\end{numpf}

\section {Change of flags}
Ideally, one would like to have a representation construction for matroids that does not depend on a choice of flag. Failing this, one might hope for nice topological relationship (e.g., a canonical homotopy equivalence) between any two homotopy representations of a given matroid. If $\Ff_1$ and $\Ff_2$ are complete flags in $M$, then, for every flat $G$ of $M$, $S_G(\Ff_1)$ and $S_G(\Ff_2)$ have the same vertex set. However, the simplicial complexes are in general not isomorphic. Examination  of small examples (for instance, when $M$ is a rank 3 uniform matroid on four elements) shows that $S_G(\Ff_1)\cap S_G(\Ff_2)$ can have ugly homotopy type.

The best result so far is:
\begin{prop} \label{prop:flag} For every rank $r$ matroid $M$ and complete flags $\Ff$ and $\Gg$ in $M$, there exists $\{C_0, \ldots, C_{r-1}\}$, with $C_i\in A_i(\Ff)$ for every $i$, such that the function $f$  sending each $ G_\epsilon$ with $G\in A_i(\Ff)$ and $\epsilon\in\{+,-\}$ to $(C_i)_\epsilon$ induces a simplicial homotopy equivalence $S_{\hat 0}(\Ff)\to S_{\hat 0}(\Gg)$.
\end{prop}

Of course, this map will not preserve the combinatorics of the homotopy arrangement.

\begin{lemma} For every rank $r$ matroid $M$ and complete flags $\Ff$ and $\Gg$ in $M$, there exists $\{C_0, \ldots, C_{r-1}\}\subset\coat(M)$ such that, for every $i\neq j$, $C_i$ and $C_j$ are in different parts of the partition of $\coat(M)$ induced by $\Ff$ and  different parts of  the partition of $\coat(M)$ induced by $\Gg$.
\end{lemma}

\begin{proof} The proof is by induction on $r$, with the case $r=1$ trivial.

Let $\Ff=\{F_0\subset\cdots\subset F_r\}$ and $\Gg=\{G_0\subset\cdots\subset G_r\}$. Consider the flag $\{G_0\cap F_{r-1}\, \ldots, G_r\cap F_{r-1}\}$ in $[\hat 0, F_{r-1}]$. By Lemma~\ref{lem:flag}, this is a complete flag in $[\hat 0, F_{r-1}]$, so there exists a unique $k$ such that $G_{k-1}\cap F_{r-1}=G_k\cap F_{r-1}$, and so $A_k(\Gg)\subset A_r(\Ff)$. Let $C_r\in A_k(\Gg)$. By the induction hypothesis applied to the flags $\Ff\backslash\{F_r\}$ and $\{G_0\cap F_{r-1}\, \ldots, G_r\cap F_{r-1}\}$ in $[\hat 0, F_{r-1}]$, 
we can choose $\{C_1, \ldots, C_{r-1}\}$ with each $C_i$ in a different part of each partition, so $\{C_1, \ldots, C_k\}$ is our desired set for $M$.
\end{proof}

\begin{numpf}{Proof of Proposition~\ref{prop:flag}} Consider the $\{C_0, \ldots, C_{r-1}\}$ of the previous  lemma. For each $i$, $C_i\in A_i(\Ff)\cap A_{i'}(\Gg)$ for some $i'$ Because each $C_i$ is in a distinct part of $\Ff$ and of $\Gg$, the cross-polytope $P$ with maximal simplices $\{\{(C_0)_{v_0},\ldots, (C_{r-1})_{v_{r-1}}\}:v_i\in\{+,-\}\mbox{ for every i}\}$ is a subcomplex of both $S_{\hat 0}(\Ff)$ and $S_{\hat 0}(\Gg)$. In fact, by Lemma~\ref{lem:intlattice}, the functions $f$ resp $f_\Gg: \{G_\epsilon:G\in\coat(M),\epsilon\in\{+,-\}\to \{(C_1)_\epsilon:,\epsilon\in\{+,-\}\}$ taking each $G_\epsilon$ in $A_i(\Ff)$ resp. $A_{i'}(\Gg)$ to $(C_i)_\epsilon$ induce retractions of $S_{\hat 0}(\Ff)$ resp. $S_{\hat 0}(\Gg)$ to $P$. 
\end{numpf}

\section{Weak maps}

\begin{defn} Let $M$ be a matroid on elements $E$. The {\bf rank} of $A\subseteq E$ is the rank of the smallest flat of $M$ containing $A$.
\end{defn}

\begin{defn} 1. Let $M$ and $N$ be matroids on elements $E$. We say there is a {\bf weak map} $M\leadsto N$ if, for every $A\subset E$, the rank of $A$ in $M$ is greater than or equal to the rank of $A$ in $N$.

2. Let $M$ and $N$ be oriented matroids on elements $E$. We say there is a {\bf weak map} $M\leadsto N$ if, for every $X\in V^*(N)$, there exists $Y\in V^*(M)$ such that $Y\geq X$.
\end{defn}

(If $M\leadsto N$ is a weak map of oriented matroids, then there is a weak map of the underlying matroids as well.)

Just as oriented matroids have topological representations, so do their weak maps:
\begin{thm} (\cite{weakmaps})  Let $M$ and $N$ be oriented matroids of the same rank on elements $E$. Then $M\leadsto N$ if and only if there is a surjective poset map $V^*(M)-\{0\}\to V^*(N)-\{0\}$ taking each $X$ to some $Y\leq X$.

If such a map exists, then it is a homotopy equivalence.
\end{thm}

Any such poset map takes an atom  $\iota^{-1}(G_\epsilon)$ to some atom $\iota^{-1}(H_\epsilon)$ with $G\subseteq H$.

Alas, there is no analogous result for homotopy representations of matroids: a weak map $M\leadsto N$ does not induce a topological map of their homotopy representations. To make this explicit, we need a new notation: $S_G({\Ff}, M)$ will denote the homotopy sphere $S_G({\Ff})$ in a homotopy representation of a matroid $M$. In what follows we will crush all hope of a  weak map $M\leadsto N$ inducing a poset map $S_{\hat 0}({\Ff}, M)\to S_{\hat 0}({\Ff}, N)$ taking each $G_\epsilon$ to some $H_\epsilon$ with $G\subseteq H$. There are two obstacles:

1. A complete flag $\Ff$ in $M$ may not be a complete flag in $N$. Thus $S_G({\Ff}, N)$ may not be defined. 

2. Even if $\Ff$ is a complete flag in both $M$ and $N$, there may not be an induced poset map with the desired properties. Consider, for instance, the weak map $M\leadsto N$ of rank 3 matroids on elements $[4]$ given by the hyperplane arrangements (actually, arrangements of equators in $S^2$)  shown in Figure~\ref{fig:weakmaps}. $\Ff=\{\hat 0<\{1\}<\{1,2\}<[4]\}$ is a complete flag in both matroids. 
\begin{figure}
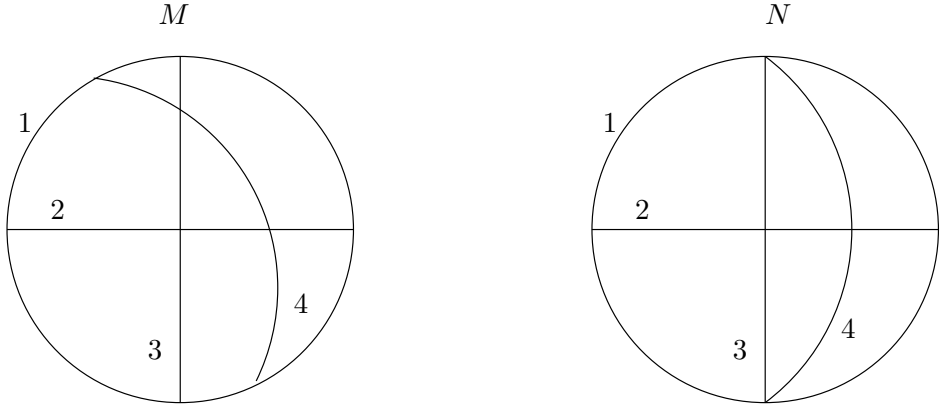
\caption{A weak map with no homotopy representation\label{fig:weakmaps}}
 \end{figure}

Any poset map $S_{\hat 0}({\Ff}, M)\to S_{\hat 0}({\Ff0}, N)$ taking each $S_G({\Ff},M)$ to some $S_H({\Ff},N)$ with $G\subseteq H$ would take $\{3,4\}_+$ to $\{1,3,4\}_+$ and take $\{1,4\}_-$ to $\{1,3,4\}_-$. Thus it would take the 1-simplex $\{\{3,4\}_+, \{1,4\}_-\}$ to something greater than both $\{1,3,4\}_+$ and $\{1,3,4\}_-$.  But there is no such simplex in $S_{\hat 0}({\Ff}, N)$.

\subsection{Homotopy tools}\label{sec:homotopy}
This section briefly reviews homotopy methods used elsewhere in the paper. For an overview of these and other homotopy methods, see~\cite{Bj}.

"Theorem A" of Quillen (\cite{Q}) was originally phrased in terms of categories. It has been recast in various ways in terms of posets and simplicial complexes (cf. \cite{Bj}). The version used in this paper is taken from~\cite{Z2}.

\begin{QTA} Let $f : P \to Q$ be a poset map.  If for all $q \in Q$, 
$\|f^{-1}(Q_{\geq q})\|$ is contractible,
then $\|f\|: \|P\| \to \|Q\|$ is a homotopy equivalence.
\end{QTA}

\begin{defn} Let $C=\{U_i:i\in I\}$ be a collection of sets. The \textbf{nerve} of $C$ is defined to be $N(C)=\{J\subseteq I:\cap_{i\in J}U_i\neq\emptyset\}$, ordered by inclusion.
\end{defn}

\begin{lemma}\label{lem:intlattice} Let $C$ and $D$ be simplicial complexes. If $\max(C)$ and $\max(D)$ have isomorphic nerves then $C\simeq D$.
\end{lemma}

\begin{proof} This is a special case of the Nerve Theorem (cf. Theorem 10.6 in~\cite{Bj})
\end{proof}

\begin{lemma}\label{carrier} Let $f:X\to Y$ be a simplical map, $A=(A_i)_{i\in I}$ and $B=(B_i)_{i\in I}$ be families of subcomplexes of $X$ resp. $Y$ such that
\begin{itemize}
\item $X=\cup_{i\in I} A_i$ and $Y=\cup_{i\in I} B_i$
\item Every nonempty intersection of elements of $A$ resp. $B$ is contractible
\item For every $J\subseteq I$, $\cap_{i\in J}A_i\neq\emptyset \Leftrightarrow \cap_{i\in J} B_i\neq\emptyset$, and
\item For every $i$, $f(A_i)\subseteq B_i$.
\end{itemize}
Then $f$ is a homotopy equivalence.
\end{lemma}

\begin{proof} Consider the commutative diagram
$$\begin{array}{ccc}
X&\stackrel{f}{\to}&Y\\
\downarrow& &\downarrow\\
N(\Gamma)&\stackrel{f}{\to}&N(\Delta)
\end{array}$$
where the vertical maps send each element to the smallest intersection containing it. Quillen's Theorem A  says that the vertical maps are homotopy equivalences, and the bottom map is an isomorphism. Thus the top map is a homotopy equivalence as well.
\end{proof}

\begin{defn}\label{defn:orderhom}  Let $P$ be a poset.

1. A function $f:P\to P$ such that $f(x)\leq x$ for all $x$ is called a \textbf{lowering homotopy}.

2. A function $f:P\to P$ such that $f(x)\geq x$ for all $x$ is called a \textbf{raising homotopy}.

An \textbf{order homotopy} is a function that is either a lowering or raising homotopy.
\end{defn}
 
\begin{thm}(cf. Corollary 10.12 in~\cite{Bj}) If $f$ is an order homotopy then $f$ induces a homotopy equivalence between $P$ and $f(P)$.
\end{thm}
 
 \section{Acknowledgement}
 Thanks to Joseph Bonin and Edward Swartz for helpful comments.
 
\bibliographystyle{alpha}
\bibliography{biblio}

\vfill

\emph{\begin{tabular}{l}
Department of Mathematics\\
Binghamton University   \\
P.O. Box 6000\\
Binghamton, NY 13902-6000\\[3pt]
laura@math.binghamton.edu
       \end{tabular}}
\end{document}